\documentclass[11pt]{amsart}
\usepackage{amsmath, amssymb}
\usepackage{graphics,color}

\newtheorem{theorem}{Theorem}[section]

\newtheorem{proposition}[theorem]{Proposition}

\newtheorem{lemma}[theorem]{Lemma}
\newtheorem{example}[theorem]{Example}

\def\qed{\hfill $\Box$\medskip}
\def\IC{{\mathbb C}}

\def\cS{{\mathcal S}}

\def\({{\left(}}
\def\){{\right)}}

\def\Diag{{\rm Diag}\,}
\def\tr{{\rm tr}\,}
\def\re{{\rm Re}\,}

\begin{document}
\baselineskip 16.3pt

\title[Preserving the higher numerical ranges]
      {Linear maps preserving the higher numerical ranges \\ of tensor products of matrices}

\author{Ajda Fo\v sner}
\author{Zejun Huang}
\author{Chi-Kwong Li}
\author{Yiu-Tung Poon}
\author{Nung-Sing Sze}
\address{Ajda Fo\v sner,
Faculty of Management, University of Primorska,
Cankarjeva 5, SI-6104 Koper, Slovenia}
\email{ajda.fosner@fm-kp.si}
\address{Zejun Huang,
Department of Applied Mathematics,
The Hong Kong Polytechnic University,
Hung Hom, Hong Kong}
\email{huangzejun@yahoo.cn}
\address{Chi-Kwong Li,
Department of Mathematics, College of William and Mary, Williamsburg, VA 23187, USA;
Department of Mathematics,
University of Hong Kong, Pokfulam, Hong Kong}
\email{ckli@math.wm.edu}
\address{Yiu-Tung Poon,
Department of Mathematics, Iowa State University,
Ames, Iowa 50011, USA}
\email{ytpoon@iastate.edu}
\address{Nung-Sing Sze,
Department of Applied Mathematics,
The Hong Kong Polytechnic University,
Hung Hom, Hong Kong}
\email{raymond.sze@polyu.edu.hk}

\date{}

\maketitle

\centerline{\bf Dedicated to Professor Natalia Bebiano on the occasion of her birthday.}

\begin{abstract}
For a positive integer $n$, let $M_n$  be the set of $n \times n$ complex matrices.
Suppose $m,n\ge 2$ are positive integers and $k\in \{1,\ldots, mn-1\}$. Denote by $W_k(X)$
the $k$-numerical range of a  matrix $X\in M_{mn}$. It is shown that a linear map
$\phi: M_{mn}\rightarrow M_{mn}$ satisfies
\begin{equation*}
W_k(\phi(A\otimes B)) = W_k(A\otimes B)
\end{equation*}
for all $A \in M_m $  and $B \in M_n $ if and only if there is a unitary $U \in M_{mn}$ such that one of the following holds.
\begin{enumerate}
\renewcommand{\labelenumi}{(\roman{enumi})}
  \item For all $A\in M_m$, $B\in M_n$,
        $$\phi(A\otimes B)=U(\varphi(A\otimes B))U^*.$$
  \item $mn = 2k$ and   for all $A\in M_m$, $B\in M_n$,
        $$\phi(A\otimes B)=(\tr(A\otimes B)/k)I_{mn}-U(\varphi(A\otimes B))U^*,$$
\end{enumerate}
where (1) $\varphi$ is the identity map $A \otimes B \mapsto A\otimes B$ or
the transposition map $A\otimes B \mapsto (A\otimes B)^t$, or
(2) $\min\{m,n\} \le 2$ and $\varphi$ has the form
$A \otimes B \mapsto A \otimes B^t$ or
$A \otimes B \mapsto A^t \otimes B$.
\end{abstract}

\medskip\noindent
{\em 2010 Math. Subj. Class.}: 15A69, 15A86, 15A60, 47A12.


\medskip

\section{Introduction and the main theorem}

Let $M_n$ ($H_n$) be the set of $n\times n$ complex (Hermitian) matrices.
For $k\in \{1,\ldots, n\}$,  the $k$-numerical range of $A\in M_n$ is defined as $$W_k(A)=\{\tr(X^*AX)/k: X \hbox{ is } n\times k, \ X^*X=I_k\}.$$
Since
$W_n(A)=\{\tr(A)/n\}$, we always assume that $k<n$ to avoid trivial consideration.
When $k=1$, we have the classical numerical range $W_1(A)=W(A)$, which has been
studied extensively; see \cite{GR,Halmos,HJ}.
Denote by $A\otimes B$ the tensor (Kronecker)
product of matrices $A\in M_m$ and $B\in M_n$. The purpose of this paper
is to characterize linear maps on $M_{mn}$  satisfying
\begin{equation}\label{eq0}
W_k(A\otimes B) = W_k(\phi(A\otimes B)) \hbox{ for all } A \in M_m, B \in M_n.
\end{equation}
The study is motivated by two areas of research.

First, in the last few decades there has been considerable interest in studying
linear preservers on matrix algebras as well
as on more general rings and operator algebras. 
For example, Frobenius \cite{Frobenious} showed that a linear operator $\phi: M_n\rightarrow M_n$ satisfies $$\det(\phi(A)) = \det(A)$$ for all $A \in M_n$
if and only if  there are $U,V\in M_n$ with $\det(UV) = 1$ such that $\phi$ has the form
\begin{equation}\label{eq01}
A \mapsto UAV \quad \hbox{ or } \quad A \mapsto UA^tV,
  \end{equation}
  where $A^t$ denotes the transpose of $A$.
Clearly, a map of the above form is linear and leaves the determinant function invariant. But it is
interesting that a linear map preserving the determinant function must be of this form. Furthermore,
in \cite{Dieudonne} Dieudonn\'{e} showed that a bijective
linear operator $\phi: M_n \rightarrow M_n$ maps the set of singular matrices into itself if and only if there are invertible $U,V \in M_n$ such that $\phi$ has
the standard form (\ref{eq01}). One may see \cite{LiPierce} and references therein for results on linear preserver problems. For more new directions and active
research on preserver problems motivated by theory and applications, one may see, for example, \cite{BresarChebotarMartindale, Molnar, Wan}.

In connection to the linear preservers of the $k$-numerical range,
Pellegrini \cite{Pellegrini} proved that every linear map $\phi:M_n\to M_n$
preserving the numerical range  is of the form
\begin{equation}
\label{a}
A\mapsto UAU\* \quad {\rm or} \quad A\mapsto UA^tU^*
\end{equation}
for some unitary $U\in M_n$. Three years later, Pierce and Watkins \cite{PierceWatkins} extended this result to the $k$-numerical ranges as
long as $n \ne 2k$. In \cite{Li} (see also \cite{Omladic}) it was shown that for $n = 2k$,
a linear map $\phi: M_n\to M_n$ preserves the
$k$-numerical range if and only if there exists a unitary $U\in M_n$ such that $\phi$ has the form (\ref{a}), or
\begin{equation}
\label{b}
A\mapsto (\tr(A)/k)I_n-UAU^* \quad {\rm or} \quad A\mapsto (\tr(A)/k)I_n-UA^tU^*.
\end{equation}
One may see \cite{CheungLiPoon,LiPoonSze} for
more information about the results on linear maps on $M_n$ which preserve the $k$-numerical range.

Another motivation of our study comes from quantum information science.
In quantum physics (e.g., see \cite{G}), quantum states are represented by
density matrices $D$ in $H_n$, i.e., positive semidefinite matrices with trace 1.
If $D$ has rank one, i.e., $D = xx^*$ for some $x \in \IC^n$ with $x^*x = 1$,
then $D$ is a pure state. Otherwise, $D$ is a mixed state, which can be written
as a convex combination of pure states. In a quantum system, every
observable corresponds to a Hermitian matrix $A$. Then
$$W(A) = \{ \tr(Axx^*): x \in \IC^n, x^*x = 1\}$$
can be viewed as the set of all possible mean measurements of quantum states.
If $A = A_1+iA_2 \in M_n$, where $A_1, A_2 \in H_n$, then
every point in $W(A)$ and be viewed as the set of all joint measurements
$x^*Ax = x^*A_1x + i x^*A_2x$ of the quantum state $x$ with respect to the two observables
associated with  $A_1$ and $A_2$, and thus $W(A)$ is the set of all possible joint
measurements.
By the convexity of $W(A)$ (e.g., see \cite{GR,Halmos,HJ}),
$$W(A) = \{ \tr(AD): D \in H_n \hbox{ is a density matrix} \}.$$
Thus, $W(A)$ can also be viewed as the set of all possible joint
measurements of two observables on mixed states.
Now, by the convexity of $W_k(A)$,
and the fact that the convex hull of the set $\{P/k: P^2 = P = P^*, \tr P = k\}$
equals the set $\cS_k$ of density matrices $D$ satisfying $I_n/k - D$ is positive
semidefinite, we have
$$
W_k(A) = \{\tr(AP)/k: P \in H_n, P^2 = P = P^*, \tr P = k\}
= \{\tr(AD): D \in \cS_k \}.$$
So, $W_k(A)$ can be viewed as the set of joint measurements of the states in
$\cS_k$ with respect to the  observables
associated with $A$.
Suppose $A\in H_m$ and $B\in H_n$ correspond to observables of two quantum systems
with quantum states $D_1 \in M_m$ and $D_2\in M_n$.
Then the tensor (Kronecker) product
$A\otimes B$ correspond to an observable on the joint (bipartite) system
with quantum states $D \in M_{mn}$ expressed as the convex combination
of uncorrelated quantum states $D_1 \otimes D_2$, where
$D_1 \in M_m$ and $D_2 \in M_n$ are quantum states (density matrices).
In this connection, we are interested in studying linear maps $\phi$ on $M_{mn}$
satisfying (\ref{eq0}).

In fact, this line of study has been carried out in \cite{FHLS,FHLS2,FHLS3,J,SLPS}.
Suppose $f(X)$ is a linear function on the matrix $X\in M_{mn}$.
It is shown in  \cite{FHLS} that the linear maps $\phi$ on $H_{mn}$   satisfying
\begin{equation}\label{eq02}
f(\phi(A\otimes B))=f(A\otimes B)
\end{equation}
for  all  $A \in H_m$ and $B \in H_n$ when $f(X)$ is the spectrum or the spectral radius of $X$. In \cite{FHLS2, FHLS3}, the authors characterized the linear maps $\phi$ on $M_{mn}$ satisfying (\ref{eq02})
for  all $ A \in M_m$   and $ B \in M_n$  when $f(X)$ is a Ky Fan norm, Schatten norm or the numerical radius of $X$.

\medskip

The following is our main result.

\begin{theorem}
\label{T1}
Let $k\in \{1,\ldots, mn-1\}$. A linear map $\phi: M_{mn}\rightarrow M_{mn}$ satisfies
\begin{equation}
\label{E1}
W_k(\phi(A\otimes B)) = W_k(A\otimes B)
\end{equation}
for all $A\in M_m$ and $B \in M_n$ if and only if there is a unitary $U \in M_{mn}$ such that one of the following holds.
\begin{enumerate}
\renewcommand{\labelenumi}{(\roman{enumi})}
  \item For all $A\in M_m$, $B\in M_n$,
      \begin{equation}\label{E2}
      \phi(A\otimes B)=U(\varphi(A\otimes B))U^*.
      \end{equation}
  \item $mn = 2k$ and   for all $A\in M_m$, $B\in M_n$,
         \begin{equation}\label{E3}\phi(A\otimes B)=(\tr(A\otimes B)/k)I_{mn}-
         U(\varphi(A \otimes B))U^*,
         \end{equation}
\end{enumerate}
where (1) $\varphi$ is the identity map $A \otimes B \mapsto A\otimes B$ or
the transposition map $A\otimes B \mapsto (A\otimes B)^t$, or
(2) $\min\{m,n\} \le 2$ and $\varphi$ has the form
$A \otimes B \mapsto A \otimes B^t$ or
$A \otimes B \mapsto A^t \otimes B$.
\end{theorem}

The proof of the theorem will be given in the next section. We will use the
following properties of the $k$-numerical range; for example, see
\cite{Halmos, Li,MF, PierceWatkins}.

\begin{proposition}
\label{P1}
Let  $A\in M_n$ and $k\in \{1,\ldots, n-1\}$.
\begin{itemize}
	\item For any $\alpha,\beta\in \mathbb C$, $W_k(A)=\{\alpha\}$   if and only if  $A=\alpha I_n$,  and $W_k(\alpha I_n+\beta A)=\alpha+\beta W_k(A)$.
	\item For any unitary $U\in M_n$, $W_k(UAU^*)=W_k(A)$.
	\item For any $s\times n$ matrix $V$ with $s\ge k$ and $VV^*=I_s$, we have $W_k(VAV^*)\subseteq W_k(A)$.
\item  $W_k(A)\subseteq \mathbb{R}$ if and only if $A$ is Hermitian.
\item If $A\in H_n$ has eigenvalues $\alpha_1\ge\cdots \ge \alpha_n$, then $$W_k(A)=[(\alpha_{n-k+1}+\cdots +\alpha_n)/k,(\alpha_1+\cdots +\alpha_k)/k].$$
    \item   $W_k(\frac{A+A^*}{2})=\re (W_k(A))\equiv \{\re (z): z\in W_k(A)\}.$
\end{itemize}
\end{proposition}

To conclude our introduction, let us point out that we consider only the bipartite case, i.e., $M_m\otimes M_n$ with integers $m,n\ge 2$.
Our proofs are rather technical and we are not able to extend them to the  multipartite systems $M_{n_1}\otimes\cdots \otimes M_{n_m}$ with 
$n_1,\ldots, n_m\ge 2$ and $m>2$. Furthermore, one can define the $k$-numerical radius of a square matrix $A\in M_n$ by
$$w_k(A)=\max\{|x|:x\in W_k(A)\}.$$ It would also be interesting to characterize the linear preservers of the $k$-numerical 
radius on the bipartite or multipartite systems. Again, it does not seem easy to apply our proofs to solve this problem.

\baselineskip 16.2pt
\section{Proof of Theorem \ref{T1}}

In the following,
denote by $E_{ij}\in M_n$, $1\le i,j\le n$, the matrix whose
$(i,j)$-entry is equal to one and all
the others are equal to zero.
Two matrices $A,B\in M_n$ are called orthogonal if $AB^*=A^*B=0$. We write $A\bot B$ to indicate that $A$ and $B$ are orthogonal. It is shown in
\cite{LiSemrlSourour} that $A\bot B$ if and only if there are unitary matrices $U,V\in M_n$ such that $UAV=\Diag(\alpha_1,\ldots,\alpha_n)$ and
$UBV=\Diag(\beta_1,\ldots,\beta_n)$ with $\alpha_i,\beta_i\geq 0$ and $\alpha_i\beta_i=0$ for $i=1,\dots,n$. The matrices $A_1,\ldots,A_s$ are said
to be pairwise orthogonal if $A_i^*A_j=A_iA_j^*=0$ for any distinct $i,j\in \{1,\ldots,s\}$. In this case, there are unitary matrices $U,V\in M_n$
such that $UA_iV=D_i$ for $i=1,\dots,s$ with each $D_i$ being nonnegative diagonal matrix and $D_iD_j=0$ for any distinct $i,j\in \{1,\ldots,s\}$. We will need the following lemmas in the proof.
The first lemma was proved in \cite{Li2}.

\begin{lemma}\label{4.1} \cite[Lemma 4.1]{Li2}
Let $k\in \{1,\ldots,n\}$ and suppose
$A \in H_n$ have
diagonal entries $a_1,\dots,a_n$ and eigenvalues $\lambda_1 \ge \cdots \ge \lambda_n$ respectively.
Then
$\sum_{j=1}^k a_{j} = \sum_{j=1}^k \lambda_j$
if and only if
$A = A_1 \oplus A_2$ where $A_1 \in H_k$ has eigenvalues $\lambda_1,\dots,\lambda_k$.
\end{lemma}

\begin{lemma}\label{le1}
Let $k\in \{1,\ldots,n\}$ and  $A,B\in H_n$ be positive semidefinite matrices. Suppose
$\tr(A)/k = \max\{x:x\in W_k(A-B)\}$. Then $A\bot B$.
\end{lemma}
\begin{proof}
Suppose $U\in M_n$ is a unitary matrix such that $U(A-B)U^*=\Diag(\lambda_1,\ldots,\lambda_n)$ with $\lambda_1\geq \cdots\geq \lambda_n$. Denote the diagonal entries of $UAU^*$ and $UBU^*$ by $a_1,\ldots,a_n$ and $b_1,\ldots,b_n$, respectively. Then $a_i,b_i$ are nonnegative for $i=1,\ldots,n$, since $A\geq 0$ and $B\geq 0$. Now
$$\sum_{i=1}^k(a_i-b_i)=\sum_{i=1}^k\lambda_i=\tr(A)$$
leads to
$$\sum_{i=1}^ka_i=\tr(A), \quad
a_{k+1}=\cdots=a_n=0 \quad \hbox{ and } \quad b_1=\cdots=b_k=0.$$
Using the fact $A\geq 0$ and $B\geq 0$ again,  $UAU^*$ and $UBU^*$ must have the form
$$UAU^*=A_1\oplus 0_{n-k},\quad UBU^*=0_{k}\oplus B_1$$
which means $A\bot B$.
\end{proof}

Denote by $\lambda_1(X) \ge \cdots \ge
\lambda_n(X)$ the eigenvalues of a Hermitian matrix
$X \in M_n$.

\begin{lemma} \label{lem0} Let $k\in \{1,\ldots, mn-1\}$ and $\phi: M_{mn}\rightarrow M_{mn}$ be a linear map satisfying (\ref{E1}). The following conditions hold.
\begin{itemize}
	\item[{\rm (a)}] $\phi(H_{mn})\subseteq H_{mn}$.
\item[{\rm (b)}] $\phi(I_{mn})=I_{mn}$.
\item[{\rm (c)}]  $\phi$ is trace-preserving. Furthermore, if $mn\ne 2k$, then  $\phi(E_{ii}\otimes E_{jj})$ is   positive semidefinite for all $1\leq i\leq m$ and $1\leq j\leq n$.
\end{itemize}
\end{lemma}

\it Proof. \rm (a) Let $A\in H_m$ and $B\in H_n$. Then
$W_k(\phi(A\otimes B))=W_k(A\otimes B)\subseteq \mathbb{R}$.
By Proposition \ref{P1}, $\phi(A \otimes B) \in H_{mn}$.
Since every $C \in H_{mn}$ is a linear combination of matrices
of the form $A\otimes B$ with $A \in H_m$ and $B \in H_n$,
we see that   $\phi$   maps $H_{mn}$  to $H_{mn}$.

\smallskip
(b) $W_k(\phi(I_{mn})) = W_k(I_{mn})=\{1\}.$  Thus, $\phi(I_{mn})=I_{mn}$.

\smallskip
(c) Let $\alpha_1\ge\cdots \ge \alpha_{mn}$ be eigenvalues of $\phi(E_{ii}\otimes E_{jj})=A_{ij}$. Since $W_k(A_{ij})=W_k(E_{ii}\otimes E_{jj})=[0,1/k]$,
we have $\alpha_1+\cdots + \alpha_k=1$ and $\alpha_{mn-k+1}+\cdots + \alpha_{mn}=0.$
If $mn=2k$, then $\tr A_{ij}=1+0=1$.

If $mn>2k$, then $\alpha_{k+1},\ldots,\alpha_{mn-k}\ge\alpha_{mn-k+1}\ge  0$ and,
thus, $\tr(A_{ij})\ge 1$. Moreover, if $\alpha_{mn-k}>0$, then $\tr(A_{ij})>1$.
On the other hand, we have
$$mn=\tr (I_{mn})=\tr\left(\phi(I_{mn})\right)
=\tr\left(\phi(\sum_{i,j} E_{ii}\otimes E_{jj})\right)
=\tr\left(\sum_{i,j}A_{ij}\right)\ge mn.$$
This yields that $A_{ij}$ is a positive semidefinite matrix with trace one.

Similarly, if $mn<2k$, then $\alpha_{mn-k+1}+\cdots+\alpha_{k}\ge 0$ and,
thus, $\tr(A_{ij})\le 1$.
Therefore, $$mn=\tr (I_{mn})=\tr(\phi(I_{mn}))
=\tr\left(\phi(\sum_{i,j} E_{ii}\otimes E_{jj})\right)
=\tr\left(\sum_{i,j}A_{ij}\right)\le mn,$$
which yields that $A_{ij}$ is a positive semidefinite matrix with trace one.

We can apply the same argument to  show that for any
orthonormal  bases $\{x_1, \ldots, x_m\}\subseteq \mathbb{C}^m$  and
$\{y_1, \ldots, y_n\}\subseteq \mathbb{C}^n$,
$\tr(\phi(x_ix_i^* \otimes y_jy_j^*)) = 1$.
Thus, $\phi$ is trace preserving for all Hermitian $A \otimes B$
and, hence, for all matrices in $M_{mn}$.
\qed

\begin{lemma} \label{lem1}  Let $k\in\{2,\ldots,N-2\}$ and $X,Y\in H_N$  with
$W_k(X)=W_k(Y)=[0,1/k]$  and $W_k(X+Y)=[0,2/k].$ If $$ X-Y = \Diag( 1 - (k-1)a, a, \ldots,a, -1-(k-1)a)$$ with $a \in [-1/k,1/k]$,
then
$$ X  = \Diag(1-(k-1)d,d, \ldots,d,-(k-1)d) \ \hbox{ and } \
 Y  = \Diag(-(k-1)d,d, \ldots,d,1-(k-1)d)$$
with $d \in \{0, 1/k\}$
so that $ X-Y  = \Diag(1, 0,\ldots,0,-1)$.
\end{lemma}

\it Proof. \rm
 Suppose $X$ has diagonal entries $x_1, \dots, x_{N}$ and $Y$ has diagonal entries $y_1, \dots, y_{N}$.
Then for any $1 = i_1 < \ldots< i_k \le N-1$,
we have
\begin{eqnarray*}
 \sum_{t=1}^kx_{i_t}   \leq  1,\quad
 \sum_{t=1}^ky_{i_t}  \geq  0, \quad
 \sum_{t=1}^kx_{i_t}-\sum_{t=1}^ky_{i_t}  =  1,
\end{eqnarray*}
which imply   $x_{i_1}+\cdots+x_{i_k}=1$ equal to the
sum of the $k$ largest eigenvalues of $X$ and $y_{i_1} +\cdots + y_{i_k}  = 0$ equal to the
sum of the $k$ smallest eigenvalues of $Y$.
Thus, applying Lemma \ref{4.1},
$x_1$ is the largest eigenvalue of $X$,  $x_2 = \cdots = x_{N-1}$
are the second largest eigenvalue of $X$,
  $y_1$ is the smallest eigenvalue of $Y$, and $y_2 = \cdots = y_{N-1}$
are the second smallest eigenvalue of $Y$. Moreover,
 $X$ and $Y$ have the form
$$X = \Diag(1-(k-1)\tilde{x},\tilde{x},\ldots,\tilde{x},-(k-1)\tilde{x}), \quad Y = \Diag(-(k-1)\tilde{y},\tilde{y},\ldots,\tilde{y},1-(k-1)\tilde{y})$$ with $\tilde{x}, \tilde{y} \in [0, 1/k].$
Hence, $X+Y = \Diag(1-(k-1)\tilde{x}-(k-1)\tilde{y}, \tilde{x}+\tilde{y}, \ldots, \tilde{x}+\tilde{y}, 1-(k-1)\tilde{x}-(k-1)\tilde{y})$
satisfies $W_k(X+Y) = [0, 2/k]$. So,
either (a) $k(\tilde{x}+\tilde{y}) = 2$, which implies that $\tilde{x} = \tilde{y} = 1/k$, or
(b) $k(\tilde{x}+\tilde{y}) = 0$, which implies that $\tilde{x} = \tilde{y} =0$.
\qed

 \begin{lemma}\label{le4}
Let $2\leq k\leq mn/2$ be an integer and $\phi: M_{mn}\rightarrow M_{mn}$ be a linear map satisfying (\ref{E1}).  Then for any orthonormal bases
$\{x_1, \dots, x_m\} \subseteq \IC^m$ and $\{y_1, \dots, y_n\} \subseteq \IC^n$, either

\begin{itemize}
\item[{\bf (1)}] there is a unitary $U \in M_{mn}$ such that
$U^*\phi(x_ix_i^*\otimes y_jy_j^*)U = x_ix_i^*\otimes y_j y_j^*$
for all $i = 1, \dots, m$, and $j = 1, \dots, n$, or
\item[{\bf (2)}] $mn = 2k$ and there is a unitary $U \in M_{mn}$ such that
$U^*\phi(x_ix_i^*\otimes y_jy_j^*)U = I_{mn}/k - x_ix_i^* \otimes y_jy_j^*$
for $i = 1, \dots, m$, and $j = 1, \dots, n$.
\end{itemize}
\end{lemma}
The proof of this lemma is rather technical. We will present it in the last part of this paper.

\vspace{0,2cm}

Denote by $\sigma(X)$ the spectrum of $X\in M_n$. The following example is useful in our proof.
\begin{example}\em \label{Ex1}  Suppose $m, n \ge 3$.
Let
$A = X \oplus O_{m-3}$ and  $B = X \oplus O_{n-3}$ with
$X= \begin{bmatrix}0 & 3 & 0 \cr 0 & 0 & 1 \cr 0 & 0 & 0 \cr\end{bmatrix}$.
Then $A \otimes B$ is unitarily similar to
$$ O_{mn-7}\oplus \begin{bmatrix}0 & 3  \cr 0 & 0 \end{bmatrix}
\oplus \begin{bmatrix}0 & 3  \cr 0 & 0 \end{bmatrix}
\oplus \begin{bmatrix}0 & 9 & 0 \cr 0 & 0 & 1  \cr 0 & 0 & 0 \cr\end{bmatrix},$$
and $A \otimes B^t$ is unitarily similar to
$$ O_{mn-7}
\oplus \begin{bmatrix}0 & 1  \cr 0 & 0 \end{bmatrix}
\oplus \begin{bmatrix}0 & 9  \cr 0 & 0 \end{bmatrix}
\oplus \begin{bmatrix}0 & 3 & 0 \cr 0 & 0 & 3 \cr 0 & 0 & 0 \cr\end{bmatrix}.$$
Consequently,
\begin{eqnarray*}
\sigma((A\otimes B+(A\otimes B)^*)/2)&=& \{-\sqrt{41/2},
   -3/2,
   -3/2,
   0,
        \ldots,
         0,
    3/2,
    3/2,
    \sqrt{41/2}\},
\\
    \sigma((A\otimes B^t+(A\otimes B^t)^*)/2)&=&\{ -9/2,
   -\sqrt{9/2},
   -1/2,
         0,
         \ldots,
    0,
    1/2,
     \sqrt{9/2},
    9/2\}.
    \end{eqnarray*}
Applying Proposition \ref{P1}, one see that
$\re (W_k(A\otimes B))\ne \re (W_k(A\otimes B^t))$,
and hence $W_k(A\otimes B) \ne  W_k(A\otimes B^t) $ for any $k\in \{1,\ldots,mn-1\}$.
\end{example}

\medskip
Now we are ready to present the proof of Theorem \ref{T1}.

\noindent{\bf Proof of Theorem \ref{T1}}.

Note that the 1-numerical range is just the classical numerical range.
The case $k=1$ has been obtained in \cite[Theorem 2.1]{FHLS3}. Since $(n-k)W_{n-k}(A)=\tr (A) -kW_k(A)$, by Lemma \ref{lem0}, we have
$$W_k(\phi(A))=W_k(A)\quad \Longleftrightarrow \quad   W_{n-k}(\phi(A))=W_{n-k}(A)\,.$$
Therefore, we can focus our proof  on $2\le k\le mn/2$, with $mn\ge 4$.

Note that
$W_k(X) = W_k(\phi(X))$, $W_k(X) = W_k(X^t)$, $W_k(X) = W_k(U^*XU)$ for any unitary
$U$ and $X$ in $M_{mn}$.
Furthermore, if $A \in M_2$, then $A = U_AA^tU_A$ for some unitary
$U_A$ depending on $A$ so that for any $B \in M_n$, the matrix
$A\otimes B$ is unitarily similar to  $A^t \otimes B$. Thus,
$W_k(A \otimes B) = W_k(A^t\otimes B) = W_k((A^t\otimes B)^t) = W_k(A \otimes B^t)$.
Similarly, if $A \in M_m$ and $B \in M_2$, then
$W_k(A\otimes B) = W_k(A\otimes B^t) = W_k(A^t \otimes B)$.
Combining the above, we get the sufficiency.

For the converse, suppose $W_k(A\otimes B) = W_k(\phi(A\otimes B))$
for all $(A,B) \in M_m\times M_n$.
Suppose $mn\ne 2k$. Then by Lemma \ref{le4}, {\bf (1)} always  holds.
So, for any Hermitian $A \in M_m$ and $B \in M_n$
with spectral decomposition $A = \sum a_i x_i x_i^*$ and $B = \sum b_j y_jy_j^*$,
we see that $\phi(A \otimes B)$ is unitarily similar to
$A\otimes B$. So, $A\otimes B$ and $\phi(A\otimes B)$ always have the
same eigenvalues. Thus, by \cite[Theorem 3.2]{FHLS},
there is a unitary $V$ such that $\phi$ has the form
$A\otimes B \mapsto V^*\varphi(A\otimes B)V$
for any Hermitian $A \in M_m$ and $B \in M_n$, where
$\varphi$ is one of the following forms:
\begin{enumerate}
	\item $A\otimes B  \mapsto  A\otimes B$, 
	\item $A\otimes B  \mapsto A\otimes B^t$,
	\item $A\otimes B  \mapsto A^t \otimes B$,
	\item $A\otimes B  \mapsto A^t \otimes B^t$.
\end{enumerate}
By linearity, the map $\phi$ can only have one of these forms on $M_{mn}$.
However, if $m, n \ge 3$, we see that $\varphi$ cannot be of the form $(2)$ or $(3)$ by Example \ref{Ex1}.
So, $\varphi$ can only be of the form $(1)$ or $(4)$. The desired conclusion holds.

\medskip
Now, suppose $mn = 2k$. We claim that either {\bf (1)} in Lemma \ref{le4} holds for any
choice of orthonormal bases $\{x_1, \dots, x_m\} \subseteq \IC^m$
and $\{y_1, \dots, y_n\} \subseteq \IC^n$, or {\bf (2)} in Lemma \ref{le4} holds for any
choice of orthonormal bases $\{x_1, \dots, x_m\} \subseteq \IC^m$
and $\{y_1, \dots, y_n\} \subseteq \IC^n$. To see this, note that the set
$$S = \{(x,y): x \in \IC^m,   y\in \IC^n, x^*x = 1 = y^*y\}$$
is path connected because the unit spheres in $\IC^m$ and $\IC^n$ are path connected.
Consider the continuous map from $S$ to reals defined by
$f(x,y) \mapsto |\det(\phi(xx^*\otimes yy^*))|$.
If {\bf (1)} holds for a pair of orthonormal bases containing $x$ and $y$,
then $f(x,y) = 0$; if {\bf (2)} holds for a pair of orthonormal bases containing
$x$ and $y$, then $f(x,y) = |\det(I_{mn}/k-E_{11}\otimes E_{11})| = (1/k)^{mn-1}(1-1/k)$.
It follows that either $f(x,y) = 0$ for all $(x,y) \in S$ so that
{\bf (1)} always holds, or $f(x,y) = (1/k)^{mn}(1-1/k)$ for all $(x,y) \in S$
so that {\bf (2)} always holds.

\medskip
If {\bf (1)} holds for all orthonormal bases
$\{x_1, \dots, x_m\} \subseteq \IC^m$ and
$\{y_1, \dots, y_n\} \subseteq \IC^n$,
then, by the argument in the case of $mn\ne 2k$, we see that
$\phi$ has the desired form. If {\bf (2)}  holds
for all orthonormal bases
$\{x_1, \dots, x_m\} \subseteq \IC^m$ and
$\{y_1, \dots, y_n\} \subseteq \IC^n$, then
compose $\phi$ with the map $X \mapsto (\tr X)I/k - X$ so that the
resulting map satisfies {\bf (1)}. The result follows.\qed

\medskip
Finally, we give the proof of Lemma \ref{le4}.

\noindent{\bf Proof of Lemma \ref{le4}}.

  By Lemma \ref{lem0} (a),
$\phi$ maps Hermitian matrices to Hermitian matrices.
We may focus on the case that  $x_i x_i^* = E_{ii}$ for $i = 1, \dots, m$
and $y_jy_j^* = E_{jj}$ for $j = 1, \dots, n$. Otherwise, replace
$\phi$ by the map
$\tilde \phi(A\otimes B) = \phi(V_1AV_1^* \otimes V_2BV_2^*)$
so that $V_1 \in M_m$ and $V_2 \in M_n$
are unitary matrices satisfying $V_1x_ix_i^*V_1^* = E_{ii}$ for $i = 1, \dots, m$,
and $V_2y_jy_j^*V_2^* = E_{jj}$ for $j = 1, \dots, n$.

We  divide the proof into three cases, namely

\medskip\centerline{
(a) $mn \ne 2k$, \quad (b) $mn = 2k$,  $m\leq 3$ and $n\leq 3$
\quad and \quad (c) $mn = 2k$, $m\geq 4$ or $n\geq 4$.}
\subsection{The case $mn \ne {2k}$}

\-

\noindent
{\em \underline{Claim 1}.}
There exists a unitary $U\in M_{mn}$ such that
$$\phi(E_{ii}\otimes E_{jj})=U(E_{ii}\otimes E_{jj})U^*$$
for all $1\leq i\leq m$ and $1\leq j\leq n$.

\vspace{0,2cm}

It suffices to prove that
\begin{equation}
\label{E2.1}
\phi(E_{ii}\otimes E_{jj})\bot \phi(E_{rr}\otimes E_{ss})
\end{equation}
for all pairs $(i,j)\ne (r,s)$ with $1\leq i,r\leq m$ and $1\leq j,s\leq n$.

First, suppose that $i=r$ or $j=s$. Considering
$$W_k(\phi(E_{ii}\otimes E_{jj})-\phi(E_{rr}\otimes E_{ss}))=W_k( E_{ii}\otimes E_{jj}-E_{rr}\otimes E_{ss})=[-1/k,1/k],$$
applying Lemma \ref{lem0} and Lemma \ref{le1}, we conclude that (\ref{E2.1}) holds.

Now, let $i\ne r$ and $j\ne s$. We may assume that $2\leq k\leq mn/2\leq mn-2$, we consider
$$W_k(\phi(E_{ii}\otimes (E_{jj}+E_{ss}))-\phi(E_{rr}\otimes(E_{jj}+ E_{ss})))=W_k( E_{ii}\otimes (E_{jj}+E_{ss})-E_{rr}\otimes (E_{jj}+E_{ss}))=[-2/k,2/k].$$
Applying Lemma \ref{le1} again, it follows that $\phi(E_{ii}\otimes (E_{jj}+E_{ss}))\bot\phi(E_{rr}\otimes(E_{jj}+ E_{ss}))$. Hence, we have (\ref{E2.1}).


\subsection{The case $mn = 2k$, $m\le 3 $ and $n\le  3$}

 \-

Since $mn = 2k$ is an even integer, without loss of generality we may assume that $n$ is even.
So it suffices to consider the cases when $n=2$ and $m\in \{2,3\}$.

\-

\noindent
{\it \underline{Claim 2}.}  Let $m\in \{2,3\}$ and
$A_i = \phi(E_{ii} \otimes (E_{11}-E_{22})) \in M_m \otimes M_2$ for $i = 1,\ldots,m$.
Then there is a unitary $U \in M_{2m}$ such that
$\phi(A_i) = U(E_{ii} \otimes (E_{11}-E_{22}))U^*$ for $i = 1,\ldots,m$.

\-

We only need to show that $A_1, \ldots, A_m$ are mutually orthogonal
and each $A_i$ has eigenvalues $1, -1, 0,\ldots,0$.
Note that $W_k(A_i) = [-1/k,1/k]$ for $i = 1,\ldots,m$.
So $\lambda_1(A_i) + \cdots + \lambda_k(A_i) = 1$
and $\lambda_{k+1}(A_i) + \cdots + \lambda_{2m}(A_i) = -1$.
Since $W_k(A_1+A_2) = [-2/k,2/k]$, we see that
$$\sum_{j=1}^k\lambda_j(A_1+A_2) = \sum_{j=1}^k (\lambda_j(A_1) + \lambda_j(A_2)) = 2$$
and
$$\sum_{j=k+1}^{2m}\lambda_j(A_1+A_2) = \sum_{j=k+1}^{2m}(\lambda_j(A_1) + \lambda_j(A_2)) = -2.$$
So, by a unitary similarity and applying Lemma \ref{4.1}, 
we may assume that
$A_1 = B_1 \oplus C_1$ and $A_2=B_2 \oplus C_2$ so that
$B_i$ has eigenvalues $\lambda_1(A_i), \ldots, \lambda_k(A_i)$
and $C_i$ has eigenvalues $\lambda_{k+1}(A_i), \ldots, \lambda_{2m}(A_i)$,
$i = 1, 2$.
As $W_k(A_1-A_2) = [-2/k, 2/k]$, we see that
$(B_1-B_2) \oplus (C_1-C_2)$ has eigenvalues $\gamma_1\ge \cdots \ge\gamma_{2m}$
such that $\gamma_1+\cdots+ \gamma_k = 2$ and
$\gamma_{k+1} + \cdots + \gamma_{2m} = -2$.
Clearly, $\gamma_1, \ldots, \gamma_k$ cannot all come from $B_1-B_2$, else,
$\gamma_1 + \cdots+ \gamma_k = \tr(B_1-B_2) = 0$.
Similarly, $\gamma_1, \ldots, \gamma_k$ cannot all come from $C_1-C_2$.
Now we distinguish two cases.

\vspace{0,2cm}

\noindent
{\it Case 1.} $m=2$. In this case we see that an eigenvalue of $B_1-B_2$ and an eigenvalue of
$C_1 - C_2$ sum up to $\gamma_1 + \gamma_2 = 2$.
Since $\lambda_1(B_1-B_2) \le \lambda_1(A_1) - \lambda_2(A_2)$ and
$$\lambda_1(C_1-C_2) \le \lambda_3(A_1) - \lambda_4(A_2),$$
we have
\begin{eqnarray*}
2 &=& \gamma_1+\gamma_2
\le \lambda_1(A_1) + \lambda_3(A_1) - \lambda_2(A_2) - \lambda_4(A_2) \\
&\le& \lambda_1(A_1) + \lambda_2(A_1) - \lambda_3(A_2) - \lambda_4(A_2) = 2.
\end{eqnarray*}
It follows that
$$\lambda_2(A_1) = \lambda_3(A_1) \quad \hbox{ and } \quad
\lambda_2(A_2) = \lambda_3(A_2).$$
Without loss of generality, assume that
$A_1$  is unitarily similar to a matrix of the form
$$\Diag(1-a,a,a,-1-a) \quad\hbox{ with } \quad a \in [-1/2,1/2].$$
By Lemma \ref{lem1}, we conclude that $a = 0$.

Similarly, we can show that $A_2$ has eigenvalues $1,-1,0,0$.
It is then easy to show that
$A_1, A_2$ are  orthogonal.

 \vspace{0,2cm}

\noindent
{\em Case 2.} $m=3$. We have two subcases.
\vspace{0,2cm}

\noindent
{\it Subcase 2.1.} An eigenvalue of $B_1-B_2$ and two eigenvalues of
$C_1 - C_2$ sum up to $\gamma_1 + \gamma_2 + \gamma_3 = 2$.
Since $\lambda_1(B_1-B_2) \le \lambda_1(A_1) - \lambda_3(A_2)$ and
\begin{eqnarray*}
\lambda_1(C_1-C_2) + \lambda_2(C_1-C_2)
&\le&
\lambda_1(C_1)+\lambda_2(C_1) - \lambda_2(C_2) - \lambda_3(C_2) \\
&=& \lambda_4(A_1) + \lambda_5(A_1) -  \lambda_5(A_2) - \lambda_6(A_2),
\end{eqnarray*}
we have
\begin{eqnarray*}
2 &=& \gamma_1+\gamma_2 + \gamma_3
\le \lambda_1(A_1) + \lambda_4(A_1) + \lambda_5(A_1) - \lambda_3(A_2) - \lambda_5(A_2) -
\lambda_6(A_2) \\
&\le& \lambda_1(A_1) + \lambda_2(A_1) + \lambda_3(A_1)
- \lambda_4(A_2) - \lambda_5(A_2) - \lambda_6(A_2) = 2.
\end{eqnarray*}
It follows that
$$\lambda_2(A_1) = \cdots = \lambda_5(A_1) \quad \hbox{ and } \quad
\lambda_3(A_2) = \lambda_4(A_2).$$

\vspace{0,2cm}

\noindent
{\em Subcase 2.2.} An eigenvalue of $C_1 - C_2$ and two eigenvalues of
$B_1-B_2$ sum up to $\gamma_1 + \gamma_2 + \gamma_3 = 2$.
Then
$$\lambda_3(A_1) = \lambda_4(A_1) \quad \hbox{ and } \quad
\lambda_2(A_2) = \cdots = \lambda_5(A_2).$$

\vspace{0,2cm}

Without loss of generality, assume that
$A_1$  is unitarily similar to a matrix of the form
$$\Diag(1-2a,a,a,a,a,-1-2a) \quad \hbox{ with } \quad a \in [-1/3,1/3].$$
By Lemma \ref{lem1}, we conclude that $a = 0$.

Now, applying the arguments to $A_2$ and $A_3$, we conclude that one of the matrices
$A_2$ and $A_3$, say, $A_2$, has eigenvalues $1, -1, 0,0,0,0$.
Then it follows from Lemma \ref{le1} that $A_1$ and $A_2$ are orthogonal. So,
we may assume that $A_1 = \Diag(1, 0,0,-1,0,0)$
and $A_2 = \Diag(0,1,0,0,-1,0)$. Note that
$$W_3(A_1+A_2-A_3) = W_3(A_1 + A_2 + A_3) = [-1,1].$$
We see that
$A_3 = B_3 \oplus C_3$ such that  $B_3$ has eigenvalues $1-2c,c,c$
and $C_3$ has eigenvalues $c,c,-1-2c$. Now, applying Lemma \ref{lem1} on $A_3$, we conclude that $A_3$ also has eigenvalues $1,-1,0,0,0,0$.
It follows from Lemma \ref{le1} that
$A_1, A_2, A_3$ are mutually orthogonal. Thus, we obtain the claim.

 \-

Using the notation  as in  Claim 2, we see that there is a unitary $U\in M_{2m}$
such that  $A_i = U(E_{ii} \otimes (E_{11}-E_{22}))U^*$.
By Lemma \ref{lem1}, for each $i = 1,\ldots,m$,
\begin{equation}\label{eq1}
\phi(E_{ii}\otimes E_{jj}) = U(E_{ii} \otimes E_{jj})U^*, \quad j = 1,2,
\end{equation}
or
\begin{equation}\label{eq2}
\phi(E_{ii}\otimes E_{jj}) = I_{2m}/k - P_i(E_{ii} \otimes E_{jj})P_i^t, \quad j = 1,2,
\end{equation}
for a suitable permutation matrix $P_i \in M_{2m}$.
Because $\sum_{i,j} \phi(E_{ii}\otimes E_{jj}) = I_{2m}$,
either (\ref{eq1}) holds for all $i = 1,\ldots,m$, or
(\ref{eq2}) holds for all $i = 1,\ldots,m$. In the latter case, the map $\tilde \phi(A)=\tr (A)/kI_{mn}-\phi(A)$ must satisfy (\ref{eq1}). This shows that $P_1 = \cdots = P_m$.


\subsection{The case $mn = 2k$, $m\geq 4$ or $n\geq 4$}

\-

\noindent
Without loss of generality, we assume $m\geq 4$.
If  $A_{ij}=\phi(E_{ii}\otimes E_{jj})$ is positive semidefinite for any
$1\leq i\leq m$ and $1\leq j\leq n$, then
applying the same arguments as in the previous
case, we conclude that $\phi$ satisfies {\bf (1)}.
Now suppose there exist some $i_0$ and $j_0$ such that $A_{i_0j_0}$ has negative eigenvalues. Without loss of generality, we assume $i_0=j_0=1$.

\-

\noindent
{\em \underline{Claim 3}.} There exists some $i\in\{1,\ldots,m\}$ such that $\phi(E_{ii}\otimes I_{n})$ has a  negative eigenvalue.

\-

For $1\leq i\leq m$, we denote the eigenvalues of  $\phi(E_{ii}\otimes I_{n})$ by $a_1(i)\geq a_2(i)\geq\cdots\geq a_{mn}(i)$.  Since $W_k(\phi(E_{ii}\otimes I_{n}))=[0,n/k]$, we have $\sum_{j=1}^ka_j(i)=n$ and $\sum_{j=k+1}^{mn}a_j(i)=0$ for $1\leq i\leq m$.
Suppose $\phi(E_{ii}\otimes I_{n})\geq 0$ for all $1\leq i\leq m$. Since $W_k(\phi(E_{11}+E_{22})\otimes I_n)=[0,2n/k]$, without loss of generality we can assume $\phi((E_{11}+E_{22})\otimes I_n)= \Diag(r_1,\ldots,r_{mn})$
with $r_1\geq r_2\geq \cdots\geq r_{k}\geq r_{k+1}=\cdots=r_{mn}=0$. Let
$\phi (E_{11} \otimes I_n)=(x_{ij})$ and $\phi (E_{22} \otimes I_n)=(y_{ij})$. Then $$\sum_{i=1}^kx_{ii}=\sum_{i=1}^ky_{ii}=n=\sum_{i=1}^ka_i(1)=\sum_{i=1}^ka_i(2).$$
By Lemma \ref{4.1}, 
$\phi (E_{11} \otimes I_n)=
X\oplus 0_k$ and $\phi (E_{22} \otimes I_n)=Y \oplus 0_k$ with $\tr X =\tr Y =n.$
Moreover, $W_k(\phi((E_{11}-E_{22})\otimes I_n))=[-n/k,n/k]$.
Thus, applying Lemma \ref{le1} we have $X\bot Y$. It follows that
$X$ is singular and $a_{k}(1)=a_{k+1}(1)=\cdots=a_{mn}(1)=0$.
Suppose $V\in M_{mn}$ is a unitary matrix such that
$$ \Diag(a_1(1),\ldots,a_{k-1}(1),0,\ldots,0)=V \phi (E_{11} \otimes
I_n)V^*=\sum_{j=1}^nV\phi(E_{11}\otimes E_{jj})V^*.$$
Denote the diagonal entries of $V\phi(E_{11}\otimes E_{jj})V^*$ by $d_{1}(j),\ldots,d_{mn}(j)$. Then $$n=\sum_{i=1}^{k-1}a_{i}(1)+a_{s}(1)=\sum_{j=1}^n(\sum_{i=1}^{k-1}d_{i}(j)+d_{s}(j))$$
for every $s\in\{k,\ldots,mn\}$
and $W_k(\phi(E_{11}\otimes E_{jj}))=[0,1/k]$ ensures that $\sum_{i=1}^{k-1}d_{i}(j)+d_{s}(j)=1$ for every $s\in\{k,\ldots,mn\}$. Applying Lemma \ref{4.1} 
again, we see that for every $j\in\{1,\ldots,n\}$, we have
$$V\phi(E_{11}\otimes E_{jj})V^*=R_j\oplus t_jI_{k+1},$$
where each eigenvalue of $R_j$ is larger than or equal to $t_j$. Further, $0\in W_k(\phi(E_{11}\otimes E_{jj}))$ implies $t_j=0$ for $1\leq j\leq n$, which contradicts with the assumption that $\phi(E_{11}\otimes E_{11})$ is not positive semidefinite.

\-

\noindent
{\em \underline{Claim 4}.} Suppose there exists $i\in\{1,\ldots,m\}$ such that the
eigenvalues of  $\phi(E_{ii}\otimes I_{n})$ are
$a_1(i)\geq a_2(i)\geq\cdots\geq a_{mn}(i)$ with $a_{mn}(i)<0$. Then
\begin{equation}\label{eq6}
a_1(i)=a_2(i)=\cdots=a_{k+1}(i).
\end{equation}
Moreover,  $\phi(E_{jj}\otimes I_{n})$ has a negative eigenvalue for every $j\in \{1,\ldots,m\}$.

\-

Without loss of generality, we assume $i=1$ and $$\phi(E_{11}\otimes I_{n})=\Diag(a_1(1),\cdots,a_{mn}(1))$$
with
\begin{equation}\label{eq11}
a_1(1)\geq a_2(1)\geq \cdots\geq a_{mn}(1),
\end{equation}
where $a_1(1),\ldots,a_{k+1}(1)$ are not identical.
Let us denote the diagonal entries of $\phi(E_{jj}\otimes I_{n})$
by $h_{1}(j),\ldots,h_{mn}(j)$ and  the diagonal entries of $U\phi(E_{jj}\otimes I_n)U^*$ by  $h_{1}(U,j),\ldots,h_{mn}(U,j)$.
Note that $a_{k}(1)$ and $a_{k+1}(1)$ must be equal.
Otherwise, by the fact that $W_k(\phi((E_{11}+E_{jj})\otimes I_n))=[0,2n/k]$,
we have $\sum_{r=1}^ka_{r}(1)=\sum_{r=1}^kh_{r}(j)=n$ for $j=2,\ldots,m$.
But then, if $Z = \sum_{j=1}^m \phi( E_{jj} \otimes I_n)$,
the leading $k\times k$ submatrix of $Z$ will have trace $mn$,
which contradicts with $W_k(\phi(I_{mn}))=\{1\}.$ Suppose $a_1(1)>a_k(1)$, i.e., there are  integers  $s,t\in\{1,\ldots,k-1\}$ such that
\begin{equation}\label{ast}a_1(1)\geq \cdots\geq a_s(1)>a_{s+1}(1)=\cdots
= a_{k+t}(1)>a_{k+t+1}(1)\geq \cdots\geq a_{mn}(1).\end{equation}
We are going to show that
\begin{equation}\label{hst}
h_1(j)=\cdots=h_s(j)=h_{k+t+1}(j)=\cdots=h_{mn}(j) \quad {\rm~for~}  j=2,\ldots,m.
\end{equation}

Let $\gamma=m/2$ when $m$ is even and $\gamma=(m+1)/2$ when $m$ is odd.
 Denote by
$$G=\phi((2(E_{11}+\cdots+E_{\gamma-1,\gamma-1})+E_{\gamma,\gamma})\otimes I_n),$$ $$G_1=\phi(( E_{11}+\cdots+E_{\gamma-1,\gamma-1} )\otimes I_n), \quad \hbox{ and } \quad
G_2=\phi(( E_{11}+\cdots+ E_{\gamma,\gamma})\otimes I_n).$$
Then
we have
\begin{eqnarray} \label{eq116}
W_k(G_1)=[0,(\gamma-1)n/k],\quad w_k(G)
 = w_k(G_1)
+ w_k(G_2),
\end{eqnarray}
where the $k$-numerical radius   $w_k(G)$, $w_k(G_1)$, and $w_k(G_2)$ are the right  end points of $W_k(G), W_k(G_1)$, and $W_k(G_2)$, respectively. Let  $U$ be a unitary such that the sum of the first $k$ diagonal entries of $UGU^*$ equals to $kw_k(G)$. Then  the sum of the first $k$ diagonal entries of $UG_iU^*$ equals to $kw_k(G_i)$ for $i=1,2$.
 We assert that the   following conditions hold.
 \begin{itemize}
\item[(a)] $\sum_{p=1}^kh_p(U,\gamma)=n$ when $m$ is even and $\sum_{p=1}^kh_p(U,\gamma)=n/2$ when $m$ is odd.
\item[(b)]  $U\phi(E_{jj}\otimes I_n)U^*=B_{j1}\oplus B_{j2}$ with $B_{j1}\in M_k$ and $\tr(B_{j1})=n$ for $j=1,\ldots,\gamma-1$.
\item[(c)] $U\phi(E_{jj}\otimes I_n)U^*=B_{j1}\oplus B_{j2}$ with $B_{j1}\in M_k$ and $\tr(B_{j2})=n$ for $j=\gamma+1,\ldots,m$.
\end{itemize}

Since $$\begin{array}{rl}(\gamma-1)n=&kw_k(G_1)=\sum_{j=1}^{\gamma-1}\sum_{r=1}^kh_r(U,j)\le \sum_{j=1}^{\gamma-1}n=(\gamma-1)n\, \mbox{, and }\\&\\
\dfrac{mn}2=&kw_k(G_2)=\sum_{j=1}^{\gamma}\sum_{r=1}^kh_r(U,j)\le \dfrac{mn}2.
\end{array}
$$
 It follows that $\sum_{p=1}^kh_p(U,j)=n$ for
$j=1,\ldots,\gamma-1$ and (a) holds.  Applying Lemma \ref{4.1}, 
we have the condition (b).

For any $j\in\{\gamma+1,\ldots,m\}$, since
$$W_k(\phi(E_{jj}\otimes I_n))= W_k(E_{jj} \otimes I_n) = [0,n/k] ,$$
the sum of any $k$ diagonal entries of $U\phi(E_{jj}\otimes I_n)U^*$ lies in
$[0, n]$. Now,  the right end point of the set
$$W_k(\phi(( E_{11}+\cdots+E_{\gamma,\gamma}+E_{jj})\otimes I_n))$$
is $1$,  and the sum of the first $k$ diagonal entries of
$U\phi(( E_{11}+\cdots+E_{\gamma,\gamma})\otimes I_n)U^*$ is $k$.
We see that $\sum_{p=1}^kh_p(U,j)=0$ and $\sum_{p=k+1}^{mn}h_p(U,j)=n$. Hence, we get (c).

 Suppose $U_1$ and $U_2$ are unitary matrices such that
$$U_1B_{11}U_1^*=\Diag(a_1(1),\ldots,a_k(1))~{\rm and} ~U_2B_{12}U_2^*=\Diag(a_{k+1}(1),\ldots,a_{mn}(1)).$$   Replace $U$ with $(U_1\oplus U_2)U$. Then the new matrix $U$ also satisfies (a), (b) and (c). Moreover,  $U$ is of   diagonal block form $U_3\oplus U_4\oplus U_5$ with $U_3\in M_s, U_5\in M_{k-t}$. Since any unitary $U_3$ and $U_5$ will yield the same summation of the first $k$ diagonal entries of $UGU^*$, we can   assume $U=I_s\oplus U_4\oplus I_{k-t}$.
Thus, for  $j=\gamma+1,\ldots,m$, it follows from  (c) that
$$h_p(j)=h_p(U,j)\leq h_q(U,j)=h_q(j) \quad {\rm~ for~ all~}  p\in\{1,\dots,s\}
\ {\rm~and~} \  q\in\{k+t+1,\ldots,mn\}.$$
On the other hand, since  $ W_k(\phi(( E_{11}+ E_{jj})\otimes I_n))=[0,2n/k]$, there
exists a unitary matrix $V$ of the form $V=I_s\oplus V_1\oplus I_{k-t}$
 such that $\sum_{p=1}^kh_p(V,j)=n$, which implies
$$h_p(j)\geq h_q(j) \quad {\rm~ for~ all~}  p\in\{1,\dots,s\} \
{\rm~and~}  q\in\{k+t+1,\ldots,mn\}.$$
  Hence, we have
\begin{equation*}
h_1(j)=\cdots=h_s(j)=h_{k+t+1}(j)=\cdots=h_{mn}(j) \quad {\rm~for~}  j=\gamma+1,\ldots,m.
\end{equation*}
Interchanging the roles of $\{2,\ldots,\gamma\}$ and $\{\gamma+1,\ldots,2\gamma-1\}$ and applying the same argument, we have (\ref{hst}).

Let $T=\{1,k+1,\ldots,mn-1\}$. Note that   $h_1(U,1)=a_1(1)>a_{mn}(1)=h_{mn}(U,1)$.  We have $$\sum_{p\in T}h_p(U,1)>\sum_{p=k+1}^{mn}h_p(U,1)=0.$$ By the fact that $h_1(U,j)=h_1(j) = h_{mn}(j)=h_{mn}(U,j)$ for
$j = 2, \dots, m$,  we have
\begin{eqnarray}\label{eq1021}
\sum_{p\in T}\left(h_p(U,1)+\sum_{j=\gamma}^{m}h_p(U,j)\right)
 &=&\sum_{p\in T}h_p(U,1)+\sum_{j=\gamma}^{m}\left(\sum_{p=k+1}^{mn}h_p(U,j)\right)\\
\nonumber&=&\sum_{p\in T}h_p(U,1)+k >k
\end{eqnarray}
which contradicts with $w_k(\phi((E_{11}+E_{\gamma,\gamma}+\cdots+E_{mm})\otimes I_n))=1$.
Hence, we get (\ref{eq6}).

 Next, suppose there is some $2\leq j\leq m$ such that $\phi( E_{jj} \otimes I_n)\geq 0$, say, $j=2$. Again, we can assume  $\phi(E_{11}\otimes I_n)=\Diag(a_1(1),\ldots,a_{mn}(1)) $
 with
 $$a_1(1)=\cdots=a_{k+t}(1)>a_{k+t+1}(1)\geq\cdots\geq a_{mn}(1),\quad 1\leq t\leq k-1$$
 and  $\phi(E_{22}\otimes I_n)=\Diag(a_1(2),\ldots,a_{mn}(2))$ with
 $$a_1(2)\geq \cdots\geq a_s(2)>a_{s+1}(2)=\cdots=a_{mn}(2)=0,\quad 1\leq s\leq k.$$

 Recall that there is a unitary matrix $U$ satisfying (a), (b), and (c).
 Suppose $U_1,U_2$ are unitary matrices   such that $U_1B_{21}U_1^*=\Diag(a_1(2),\ldots,a_k(2))$ and $U_2B_{12}U_2^*=\Diag(a_{k+1}(1),\ldots,a_{mn}(1))$.  Replace $U$ with $(U_1\oplus U_2)U$. Then the new matrix $U$ also satisfies  (a), (b), and (c). Moreover,
 \begin{eqnarray*}
 U\phi(E_{11}\otimes I_n)U^*&=&\Diag(a_1(1),\ldots,a_{mn}(1)),\\
 U\phi(E_{22}\otimes I_n)U^*&=&\Diag(a_1(2),\ldots,a_{mn}(2))
 \end{eqnarray*}
 which implies that $U$ is of the form $U=U_3\oplus U_4\oplus U_5$ with $U_3\in M_s, U_5\in M_{k-t}$. We can assume $U_3=I_s$ and $U_5=I_{k-t}$.

For any given $j\in\{\gamma+1,\ldots, m\}$, we denote by $\alpha_1\geq\cdots\geq\alpha_k $ the eigenvalues of $B_{j1}$ and $ \beta_1\geq\cdots\geq\beta_k$ the eigenvalues of $B_{j2}$. Then $\beta_k\geq \alpha_1$. By $W_k(\phi((E_{22}+E_{jj})\otimes I_n))=[0,2n/k]$, there is a unitary $V$ such that
\begin{eqnarray}\label{eq1019}
V(B_{21}\oplus B_{22})V^*=Y\oplus 0~{\rm and}~   V(B_{j1}\oplus B_{j2})V^*=Z_{1}\oplus Z_{2}
\end{eqnarray}
with $Y,Z_{1}\in M_k$, $\tr(Y)=\tr(Z_{1})=n$. Suppose $W$ is a unitary matrix such that $WYW^*=B_{21}$. Replace $V$ with $(W\oplus I_k)V$. Then we still have (\ref{eq1019}) with $Y=B_{21}$. Moreover, $V$ is of the form $V=V_1\oplus V_2$ with $V_1\in M_s$ and we can assume $V_1=I_s$. Partition $B_{j1}$ and $Z_1$    as
 $$B_{j1}=\begin{bmatrix}
C_{11}&C_{12}\\
C_{21}&C_{22}\\
\end{bmatrix},\quad Z_1=\begin{bmatrix}
D_{11}&D_{12}\\
D_{21}&D_{22}\\
\end{bmatrix} $$ with $C_{11},D_{11}\in M_s$. We can rewrite the second equation in (\ref{eq1019}) as
\begin{eqnarray*}
\begin{bmatrix}
I_s&   \\
&V_2\end{bmatrix}
\begin{bmatrix}
C_{11}&C_{12}&\\
C_{21}&C_{22}&\\
& & B_{j2}\end{bmatrix}\begin{bmatrix}
I_s&   \\
&V_2\end{bmatrix}^*=\begin{bmatrix}
D_{11}&D_{12}&\\
D_{21}&D_{22}&\\
& & Z_{2}\end{bmatrix}.
\end{eqnarray*}
It is clear that $D_{11}=C_{11}$. Since $\tr(C_{11}+D_{22})=n$ equals to the sum of the $k$ largest eigenvalues of $B_{j1}\oplus B_{j2}$, we see that $\tr (D_{22})=\sum_{p=1}^{k-s}\beta_p$.
Applying Lemma \ref{4.1}, 
we have $D_{12}=D_{21}^*=0$, which implies $C_{12}= C_{21}^*=0$ and $\sigma(C_{11})=\{\alpha_1\ldots,\alpha_s\}$. It follows that $\alpha_1=\cdots=\alpha_s=\beta_{k-s+1}=\cdots=\beta_k$ and $C_{11}=\alpha_1 I_s$.

Similarly, considering $W_k(\phi((E_{11}+E_{jj})\otimes I_n))=[0,2n/k]$, there is a unitary $\tilde{V}$ such that
\begin{eqnarray}\label{eq1020}
\tilde{V}(B_{11}\oplus B_{12})\tilde{V}^*=a_1(1)I_k\oplus \tilde{Y}~{\rm and}~   \tilde{V}(B_{j1}\oplus B_{j2})\tilde{V}^*=\tilde{Z}_{1}\oplus \tilde{Z}_{2}
\end{eqnarray}
with $\tilde{Z}_{1}\in M_k$, $ \tr(\tilde{Z}_{1})=n$. Suppose $\tilde{W}$ is a unitary matrix such that $\tilde{W}\tilde{Y}\tilde{W}^*=B_{12}$. Replace $\tilde{V}$ with $(I_k\oplus \tilde{W})\tilde{V}$. Then we still have (\ref{eq1020}) with $\tilde{Y}=B_{12}$. Moreover, $\tilde{V}$ is of the form $\tilde{V}=\tilde{V}_1\oplus \tilde{V}_2$ with $V_2\in M_{k-t}$ and we can assume $V_2=I_{k-t}$. Partition $B_{j2}$ and $\tilde{Z}_2$ as
 $$B_{j2}=\begin{bmatrix}
R_{11}&R_{12}\\
R_{21}&R_{22}
\end{bmatrix},\quad \tilde{Z}_2=\begin{bmatrix}
S_{11}&S_{12}\\
S_{21}&S_{22}\\
\end{bmatrix} $$ with $R_{11}, S_{11}\in M_t$. We can rewrite the second equation in (\ref{eq1020}) as
\begin{eqnarray*}
\begin{bmatrix}
\tilde{V}_1&\\
 &I_{k-t}\end{bmatrix}
\begin{bmatrix}
C_{11}& &&\\
 &C_{22}& &\\
& & R_{11}&R_{12}\\
& &R_{21}&R_{22}\end{bmatrix}\begin{bmatrix}
\tilde{V}_1&\\
 &I_{k-t}\end{bmatrix}^*=\begin{bmatrix}
\tilde{Z}_1 & &\\
 & S_{11}&S_{12}\\
&S_{21}&S_{22}\end{bmatrix}.
\end{eqnarray*}
Since $\tr(\tilde{Z}_1)=n$ is the sum of the $k$ largest eigenvalues of $B_{j1}\oplus B_{j2}$, which equals to the sum of the  $k$ largest eigenvalues of $C_{11}\oplus C_{22}\oplus R_{11}$, we see that the eigenvalues of $R_{11}$ are also the  $t$ largest eigenvalues of $B_{j2}$. Hence, we have $R_{12}= R_{21}^*=0$ and  $R_{22}=\beta_{k}I_{k-t}=\alpha_{1}I_{k-t}$.

So we have
\begin{equation}\label{eq31}
h_1(j)=h_1(U,j)=h_{mn}(U,j)=h_{mn}(j) ~{\rm for}~  j=\gamma+1,\ldots, m.
  \end{equation}
  Similarly, (\ref{eq31}) holds  for $j=2,\ldots,\gamma $.

   Again, we have (\ref{eq1021}), which contradicts with $w_k(\phi((E_{11}+E_{\gamma,\gamma}+\cdots+E_{mm})\otimes I_n))=1$.
      Therefore, $\phi(E_{jj}\otimes I_n)$ has a negative eigenvalue for every $1\leq j\leq m$.

\-

\noindent
{\em \underline{Claim 5}.} For any $1\leq i\leq m$ and $1\leq j\leq n$, the largest eigenvalue of $\phi(E_{ii}\otimes E_{jj})$ is $1/k$ and, hence, $\frac{1}{k}I_{mn}-\phi(E_{ii}\otimes E_{jj})\geq 0$.

\-

Given any $1\leq i\leq m$,
by the previous claims, we can assume
 $$\phi(E_{ii}\otimes I_n)=\Diag(a_1(i),\ldots,a_{mn}(i))$$
 with $a_1(i)=\cdots=a_{k+1}(i)\geq\cdots\geq a_{mn}(i)$ and $a_{mn}(i)<0$.  Denote by $d_1(i,j),\ldots,d_{mn}(i,j)$ the diagonal entries of $\phi(E_{ii}\otimes E_{jj})$. Then
 $$\sum_{u\in T}d_u(i,1)=\cdots=\sum_{u\in T}d_u(i,n)=1$$ for any
 $T\subseteq\{1,\ldots,k+1\}$ with $|T|=k$. It follows that $d_u(i,j)=1/k$ for all $1\leq u\leq k+1$ and $1\leq j\leq n$.

 Applying Lemma \ref{4.1}, 
 each $\phi(E_{ii}\otimes E_{jj})$ is of the form $$\Diag(d_1(i,j), \ldots, d_{k+1}(i,j))\oplus X(i,j),$$
  where the largest eigenvalue of $X(i,j)$ is less than or equal to $1/k$. Thus, we get the claim.

\-

Now, let $\psi(A\otimes B)=(\tr(A\otimes B)/k)I_{mn}-\phi(A\otimes B)$. Then $W_k(\psi(A\otimes B))=W_k(A\otimes B)$ for all $A\in  H_m$, $B\in  H_n$, and $\psi(E_{ii}\otimes E_{jj})\geq 0$ for all $1\leq i\leq m$, $1\leq j\leq n$.
Applying the same arguments as in the first case on $\psi$, we conclude that
$\psi$ satisfies {\bf (1)} and, hence, $\phi$ satisfies {\bf (2)}. The proof is completed.\qed

\newpage
\section*{Acknowledgment}

This research was supported by a Hong Kong RGC grant PolyU 502411 with Sze as the PI and Poon as the Co-I.
The grant also supported the post-doctoral fellowship of Huang
and the visit of Fo\v{s}ner to the Hong Kong Polytechnic University in the summer of 2012.
She gratefully acknowledged the support and kind hospitality from the host university.
Fo\v sner was supported by the bilateral research program between Slovenia and US (Grant No. BI-US/12-13-023).
Li was supported by a RGC grant and a USA NSF grant; this research was done when he was
a visiting professor of the University of Hong Kong in the spring of 2012; furthermore,
he is an honorary professor of Taiyuan University of Technology (100 Talent Program scholar),
and an honorary professor of the  Shanghai University. Poon was supported by a RGC grant and a USA NSF grant.



\end{document}